\documentclass[a4paper,12pt]{amsart}
\usepackage[top=1.5in, bottom=1in, left=0.9in, right=0.9in]{geometry}

\usepackage[utf8]{inputenc}
\usepackage[all]{xy}
\usepackage{color}
\usepackage{amsmath, mathtools}
\usepackage{hyperref}
\hypersetup{colorlinks=true}
\usepackage{amssymb}
\usepackage{amsfonts}
\usepackage{graphicx}
\usepackage{mathrsfs}
\usepackage{amscd}
\usepackage[all]{xy}
\usepackage{hyperref}
\usepackage{amsthm}
\usepackage{cleveref}
\usepackage{verbatim}
\usepackage{amscd}
\usepackage{mathtools}
\usepackage{comment}

\newtheorem{thm}{\bf Theorem}[section]

\newtheorem{theoremx}{\bf Theorem}
\newtheorem{definitionx}[theoremx]{\bf Definition}
\newtheorem{conjecturex}[theoremx]{\bf Conjecture}

\newtheorem{lemma}[thm]{\bf Lemma}
\newtheorem{cor}[thm]{\bf Corollary}
\theoremstyle{definition}
\newtheorem{definition}[thm]{\bf Definition}
\theoremstyle{remark}
\newtheorem{remark}[thm]{\bf Remark}

\newtheorem{example}[thm]{\bf Example}

\newtheorem{notation}[thm]{\bf Notation}
\newtheorem{problem}[thm]{\bf Problem}

\numberwithin{equation}{section}

\newcommand{\HH}[3]{\operatorname{H}^{#1}_{#2}\left(#3\right)}

\DeclareMathOperator{\height}{{ht}}

\DeclareMathOperator{\depth}{{depth}}

\DeclareMathOperator{\Min}{{Min}}

\DeclareMathOperator{\Hom}{{Hom}}

\DeclareMathOperator{\reg}{{reg}}

\DeclareMathOperator{\gr}{{gr}}

\def\ls{\leqslant}
\def\gs{\geqslant}

\def\f0{\mathbf{0}}
\def\f1{\mathbf{1}}

\def\fv{\mathbf{v}}
\def\fx{\mathbf{x}}
\def\fy{\mathbf{y}}
\def\fb{\mathbf{b}}
\def\fc{\mathbf{c}}
\def\fp{\mathbf{p}}

\def\fa{\mathbf{a}}
\def\fc{\mathbf{c}}

\def\fm{\mathfrak{m}}

\def\fp{\mathfrak{p}}

\def \cJ{\mathcal J}

\def \QQ{\mathbb Q}

\def \RR{\mathbb R}

\def \FF{\mathbb F}
\def \ZZ{\mathbb Z}

\def \C{\mathcal C}

\def \R{\mathcal R}

\def \cJ{\mathcal J}

\begin{document}

\title[Splittings and Symbolic Powers of  square-free monomial Ideals]{Splittings and  Symbolic powers of square-free monomial Ideals}

\author[Jonathan Monta\~no]{Jonathan Monta\~no$^{*}$}
\address{Jonathan Monta\~no \\ Department of Mathematical Sciences  \\ New Mexico State University  \\PO Box 30001\\Las Cruces, NM 88003-8001}
\thanks{$^{*}$ The first author was supported by a Simons Travel Grant from the Simons Foundation and AMS.}
\email{jmon@nmsu.edu}

\author[Luis N\'u\~nez-Betancourt]{Luis N\'u\~nez-Betancourt$^{+}$}
\address{Luis N\'u\~nez-Betancourt \\ Centro de Investigaci\'on en Matem\'aticas\\ Guanajuato, Gto. 36023, M\'exico}
\thanks{$^{+}$ The second author was partially supported by NSF Grant 1502282 and CONACyT Grant 284598.}
\email{luisnub@cimat.mx}

\subjclass[2010]{Primary 	 ; Secondary  .}

  \begin{abstract}
We study the symbolic powers of square-free monomial ideals via symbolic Rees algebras and methods in prime characteristic. In particular, we prove that the symbolic Rees algebra and the symbolic associated graded algebra are split with respect to a morphism which resembles the Frobenius map and that exists in all characteristics. Using these methods, we recover a result by  Hoa and Trung which states that the normalized  $a$-invariants and the Castelnuovo-Mumford regularity of the symbolic powers converge. In addition, we  give a sufficient condition for the equality of the ordinary and symbolic powers of this family of ideals, and  relate it to Conforti-Cornu\'ejols conjecture. Finally, we interpret this condition in the context of linear optimization.
  \end{abstract}


\maketitle
\section{Introduction}

Symbolic powers of ideals
have been studied intensely over the last two decades (see   \cite{SurveySymbPowers} for a recent survey). Particular attention has been given to square-free monomial ideals, as in this setting methods from combinatorics, convex geometry, and linear optimization can be utilized to investigate properties of symbolic powers. For instance, the Cohen-Macaulay property of all symbolic powers of a square-free monomial ideal is characterized in terms  of the combinatoric structure of its underlying simplicial complex \cite{TT,TT2,Varbaro}. In addition, the symbolic Rees algebra of  monomial ideals is Noetherian \cite{LyubeznikAriRank, herzoga2007symbolic}. We note that this phenomenon  does not hold for an arbitrary ideal in a polynomial ring \cite{roberts1985prime}.

In this article, we propose a technique to deal with questions about symbolic powers of  square-free monomial ideals.
Specifically, we study  the symbolic Rees algebras of an square-free monomial ideal via methods in prime characteristic.
This combination, to the best of our knowledge, has not been previously used in combinatorial commutative algebra.
In order for our results to hold over fields of arbitrary  characteristic, we consider the map that raises every monomial to a power and resembles the  Frobenius map.
Considering this map, we obtain  that the symbolic Rees algebra and the symbolic associated graded algebra are split in this general context (see Theorem \ref{ThmGradedSpl}). In particular, these symbolic algebras are $F$-pure in prime characteristic (see Corollary \ref{Fpure}).

Motivated by the behavior of the  Castelnuovo-Mumford regularity for powers of ideals \cite{CHT,KV}, Herzog, Hoa,  and Trung \cite{HHT} asked whether  the limit  $$\lim\limits_{n\to\infty}\frac{\reg(R/I^{(n)})}{n}$$ exists for every homogeneous ideal $I$ in a polynomial ring $R$. It is known that the function $\reg(R/I^{(n)})$ is bounded by a linear function on $n$ if $I$ is a monomial ideal.  This follows because the regularity of a monomial ideal is bounded by the degree of the least common multiple of the generators  \cite{BHMultRes,HTMathZ}.

Hoa and Trung showed that the  limit above exists for square-free monomial ideals.  In fact, they showed a stronger version for the $a$-invariants \cite[Theorems 4.7 and 4.9]{LT10}. As a first consequence of our methods, we recover this result, providing an alternative proof. 

\begin{theoremx}[{see Theorem \ref{limAinv} and Corollary \ref{limReg}}]\label{MainLim}
Let $I$ be a square-free monomial ideal.
Then
$$
\lim\limits_{n\to\infty}\frac{a_i(R/I^{(n)})}{n}
$$
exists for every $0\ls i\ls \dim R/I$. In particular, 
$$
\lim\limits_{n\to\infty}\frac{\reg(R/I^{(n)})}{n}
$$
exists.
\end{theoremx}

The function $\reg(R/I^{(n)})$ is in fact a linear quasi-polynomial  if $I$ is a monomial ideal \cite{HHT}. In addition, as a consequence of our previous result, we recover  properties of this quasi-polynomial in Corollary \ref{CorQP}. 
Along the way towards proving Theorem \ref{MainLim}, we showed that $$a_i(R/I^{(n)})\gs m a_i(R/I^{(\lceil\frac{n}{m}\rceil)})$$ 
for every  $n,\, m\in \ZZ_{>0}$.
In addition, we showed that  
$$\depth(R/I^{(n)})\ls \depth(R/I^{(\lceil\frac{n}{m}\rceil)})$$ for every  $n,\, m\in \ZZ_{>0}$ (see Theorem \ref{ThmDivision}). Similar results were previously obtained for the Stanley depth of symbolic powers \cite{StanleyDepth}.

Conforti and  Cornu\'ejols \cite{CC} made a conjecture in the context of linear optimization. This conjecture was translated as a characterization of the set of square-free monomials ideals whose symbolic and ordinary powers are equal \cite {GRV,GVV}.
The following definition is needed to state this conjecture.

\begin{definitionx}  
\rm  A square-free monomial ideal $I$ of height $c$ is \it K\"onig \rm if there exists a
regular sequence of monomials in $I$ of length $c$.  The ideal $I$ is said to be  \it packed \rm if every ideal
obtained from $I$ by setting any number of variables equal to $0$ or $1$ is K\"onig.
\end{definitionx}

The Conforti-Cornu\'ejols conjecture 
can be stated as follows.

\begin{conjecturex}[\cite{CC}]\label{ConjPackingProb}
A square-free monomial ideal $I$ is packed if and only if $I^n=I^{(n)}$ for every $n\in\ZZ_{>0}.$
\end{conjecturex}

We point out that one direction of the conjecture is already known. Explicitly,   if $I^n=I^{(n)}$ for every $n\in\ZZ_{>0}$, then $I$ is packed.

A related question due to Huneke, and brought to our attention by  H\`a\footnote{ BIRS-CMO workshop on {\it Ordinary and Symbolic Powers of Ideals} Summer of 2017,  Casa Matem\'atica Oaxaca, Mexico.},  asks whether  there exists a number $N$, in terms of $I$, such that if $I^n=I^{(n)}$ for $n\ls N$, then the equality holds for every $n\gs 1$. If the answer to this question is $\height(I)$, then Conforti-Cornu\'ejols holds. 
This is expected as a similar property is known for  integral closures of ordinary powers. Specifically, it suffices to verify the equality up to   the analytic spread of $I$ minus one \cite{Singla} (see also  \cite{RRV}).  A uniform positive answer to Huneke's question can be deduced from \cite[Theorem 5.6]{herzoga2007symbolic}; if $d=\dim R$, this result guarantees  that  $N=\lfloor\frac{(d+1)^{(d+3)/2}}{2^d}\rfloor$ works for any $I$.


In our next main result, we provide a sharper $N$.

 \begin{theoremx}[{see Theorem \ref{allPowMon}}]\label{MainAllPowMon}
 Let $I$ be a square-free monomial ideal and $\mu(I)$ its minimal number of generators.
Then  $I^n=I^{(n)}$ for every $n\ls \lceil\frac{ \mu(I)}{2}\rceil$ if and only if $I^n=I^{(n)}$ for every $n\in\ZZ_{>0}$.
 \end{theoremx}

 In Example \ref{ExSharp} we show that the bound given in Theorem \ref{MainAllPowMon} is sharp. The previous result gives a finite algorithm to verify Conjecture \ref{ConjPackingProb} for a specific square-free monomial ideal.
We refer to the work of Gitler, Valencia, and Villarreal \cite[Remark 3.5]{GVV} for a different algorithm to verify Conjecture  \ref{ConjPackingProb}. The referee pointed out that another finite algorithm can be obtained from an equivalent conjecture to  Conjecture \ref{ConjPackingProb} \cite[Conjecture 5.7]{HaTrung}. Indeed, this  conjecture  states that if  a clutter satisfies  the {\it packing property} then it satisfies the {\it integer round-down property}.  On the other hand, the integer round-down property can be checked in a finite number of  steps \cite[Corollary 2.3]{BT82}.

In Section \ref{SecOptimization}, we recall the ideas from linear optimization that originally gave rise to Conjecture \ref{ConjPackingProb}. In particular, we translate 
Theorem \ref{MainAllPowMon} to this context in Theorem \ref{MFMC}, showing that the Max-Flow-Min-Cut property of clutters can be verified with a finite process (see \cite[Corollary 2.3]{BT82} for a related result).

\section{Notation}\label{prelim}
In this section we set up the notation  used throughout the entire manuscript. We assume $R=k[x_1,\ldots, x_d]$ is a standard graded polynomial ring over the field $k$ and $\fm=(x_1,\ldots,x_d)$. The  ideal $I\subseteq R$ is assumed to be a monomial ideal.

For a fixed $m\in\ZZ_{>0}$ we set  $R^{1/m}=k[x^{1/m}_1,\ldots, x^{1/m}_d]$. We denote by $I^{1/m}$ the ideal of $R^{1/m}$ generated by $\{f^{1/m}\mid f\in  I \hbox{ monomial} \} $.

We note that $R$ and $R^{1/m}$ are isomorphic as rings. 
Then the category of $R$-modules is naturally equivalent to the category of $R^{1/m}$-modules.
In addition, $I$ corresponds to $I^{1/m}$ under this isomorphism.  For an $R$-module, $M$, we denote by  $M^{1/m}$ the corresponding $R^{1/m}$-module. Given that $R\subseteq R^{1/m}$, $M^{1/m}$ obtains a structure of $R$-module via restriction of scalars.  If $A=k[x_1,\ldots,x_d,t]$, where $t$ is an extra variable, we also use the  notation  $A^{1/m}=k[x^{1/m}_1,\ldots, x^{1/m}_d,t^{1/m}]$.
We often refer to the containment $R\subseteq R^{1/m}$ and  $A\subseteq A^{1/m}$.

For a vector $\fa=(a_1,\ldots, a_d)\in \QQ^d_{\gs 0}$ we denote by $\fx^\fa$ the monomial $x_1^{a_1}\cdots x_d^{a_d}$.

\begin{definition}
Given $n\in \ZZ_{>0}$, we denote by $I^{(n)}$ the {\it {$n$th}-symbolic power of $I$}:
$$I^{(n)}=\bigcap_{\fp\in \Min(R/I)}I^nR_\fp\cap R.$$ 
\end{definition}

We now consider algebras associated to ordinary and symbolic powers of ideals.

\begin{definition} \label{blowup}
We consider the  following graded algebras.
\begin{enumerate}
\item[(i)]  The {\it  Rees algebra} of $I$: $\R(I)=R[It]=\oplus_{n\gs 0}I^nt^n\subseteq A$.
\item[(ii)] The {\it associated graded algebra} of $I$: $\gr_I(R)=\oplus_{n\gs 0}I^n/I^{n+1}$.
\item[(iii)] The  {\it symbolic Rees algebra} of $I$: $\R^s(I)=\oplus_{n \gs 0}I^{(n)}t^n\subseteq A$.
\item[(iv)] The {\it  symbolic associated graded algebra} of $I$: $\gr^s_I(R)=\oplus_{n\gs 0}I^{(n)}/I^{(n+1)}$.
\end{enumerate} 
\end{definition}

\section{Castelnuovo-Mumford regularity and $a$-invaratians}

In this section we study the graded structure of  symbolic powers. In particular, we prove Theorem \ref{MainLim}. The techniques here are inspired by methods in prime characteristic used to bound $a$-invariants of  $F$-pure graded rings \cite{HRFpurity,DSNBFpurity}.
 In particular,  the results proved in this section are motivated by the fact that the symbolic Rees and associated graded algebra are $F$-pure for every prime characteristic (see Corollary \ref{Fpure}). 

We now recall the definition of $a$-invariants and  Castelnuovo-Mumford regularity in terms of local cohomology. We refer to Broadmann and Sharp's book \cite{BroSharp} on local cohomology for more details about this subject. 
If $M$ is a graded $R$-module, we denote by $$a_i(M)=\max\{j\mid \HH{i}{\fm}{M}_j\neq 0\}$$ the {\it $i$-th $a$-invariant} of $M$ for $0\ls i\ls \dim M$. The {\it Castelnuovo-Mumford regularity} of $M$ is defined as $$\reg(M)=\max\{a_i(M)+i\}.$$


We now define a splitting from $R^{1/m}$ to $R$
 which is inspired by the trace map in prime characteristic. In fact, these maps are the same if $k=\FF_p$ and $m=p$.

\begin{definition}\label{defSplitting}
For $m\in \ZZ_{>0}$, we define  the $R$-homomorphism  $\Phi^R_m:R^{1/m}\to R$ induced by 
$$
\Phi^R_m(\fx^{\fa/m})=
\begin{cases} 
     \fx^{\fa/m} & \fa \equiv \mathbf{0}\, (\bmod\, m);\\
      0 & \hbox{otherwise,}
\end{cases}.
$$
\end{definition}
\noindent where $\fa\in  \ZZ_{\gs 0}^d$ and $\mathbf{0}=(0,\ldots, 0)\in \ZZ_{\gs 0}^d$. 

We note that $\Phi^R_m$ restricted to $R$ is the identity. Then  $R$ is isomorphic to a direct summand of $R^{1/m}$. We also consider the analogous map $\Phi^A_m:A^{1/m}\to A$.
We now show that our splitting is compatible with symbolic powers; for this  we need the following remark.

\begin{remark}\label{inclusionSymbol}
Let $I\subseteq R$ be a square-free monomial and  $Q_1,\ldots, Q_s$ its minimal primes.
We have  
\begin{align*}
I^{(n)}=Q^n_1\cap \ldots \cap Q^n_s &\subseteq 
(Q^{nm}_1)^{1/m}\cap \ldots \cap (Q^{nm}_s)^{1/m}
\\ &=(Q^{nm}_1\cap \ldots \cap Q^{nm}_s)^{1/m}=(I^{(nm)})^{1/m}.
\end{align*}

\end{remark}

 \begin{lemma}\label{LemmaStable}
Let $I$ be a square-free  monomial ideal.
 Then $$\Phi_m^R\big(( I^{(n m+j)})^{1/m}\big)= I^{(n+1)}$$ for every $n\in \ZZ_{\gs 0}$, $m\in \ZZ_{>0}$, and $1\ls j\ls m$.
 \end{lemma}
 \begin{proof}
From Remark \ref{inclusionSymbol} it follows  that $ I^{(n+1)}\subseteq ( I^{((n+1) m)})^{1/m} \subseteq  (I^{(n m+j)})^{1/m}$. Since $\Phi_m^R$ is $R$-linear, we conclude that $I^{(n+1)}\subseteq \Phi_m^R\big(( I^{(n m+j)})^{1/m}\big)$.

We now focus on the other containment. First  we prove our claim when $I$ is a prime monomial ideal   $Q=(x_{i_1},\ldots,x_{i_\ell})$ for some $1\ls i_1<\cdots <i_\ell\ls d$. 
 In this case, we have that 
 $( Q^{(n m+j)})^{1/m}$ is generated as a $k$-vector space by 
 $$\{ (\fx^{\fa})^{1/m}\mid  a_{i_1}+\ldots+a_{i_\ell}\gs nm+j \}
 $$
Let $(\fx^{\fa})^{1/m} \in ( Q^{(n m+j)})^{1/m}$ such that $\Phi_m^R( (\fx^{\fa})^{1/m})\neq 0$, then
$\fa/m\in \ZZ_{\gs 0}^d$ and  $a_{i_1}+\ldots+a_{i_\ell}\gs nm+j$.
Set $\fb=(b_{i_1},\ldots, b_{i_\ell})=\fa/m$, since
$ b_{i_1}+\ldots+b_{i_\ell}\gs n+\frac{j}{m}$, we have $ b_{i_1}+\ldots+b_{i_\ell}\gs n+1$ and hence $\Phi_m^R((\fx^{\fa})^{1/m} )\in Q^{(n+1)}$ as desired. Now, let $I$  be an arbitrary square-free monomial ideal and  let $Q_1,\ldots,Q_s$ be its minimal primes. Then, 
$
( I^{(n m+j)})^{1/m}
=( Q_1^{(n m+j)})^{1/m} \cap \cdots\cap ( Q_s^{(n m+j)})^{1/m}.
$
Therefore,
\begin{align*}
\Phi_m^R\big( (I^{(n m+j)})^{1/m}\big) &\subseteq \Phi_m^R\big( ( Q_1^{(n m+j)}\big)^{1/m} \big)\cap \cdots\cap \Phi_m^R\big(( Q_s^{(n m+j)})^{1/m} \big)\\
&\subseteq Q^{(n+1)}_1\cap\cdots\cap Q^{(n+1)}_s=I^{(n+1)},
\end{align*}
hence the result follows.
 \end{proof}

As a consequence of the previous lemma we obtain the following relations on depths and $a$-invariants of symbolic powers; 
 these relations are key ingredients in the proof of Theorem \ref{MainLim}.
 
 \begin{thm}\label{ThmDivision}
 Let $I$ be a square-free  monomial ideal and $n,\, m\in \ZZ_{>0}$. Then
 \begin{enumerate}
 \item $\depth(R/I^{(n)})\ls \depth(R/I^{(\lceil\frac{n}{m}\rceil)})$.
 \item $a_i(R/I^{(n)})\gs m a_i(R/I^{(\lceil\frac{n}{m}\rceil)})$ for every  $0\ls i\ls \dim R/I$. 
 \end{enumerate}
 \end{thm}
\begin{proof}
By Lemma \ref{LemmaStable} the natural map $\iota :R/I^{(n+1)}\to R^{1/m}/(I^{(nm+j)})^{1/m}$ splits for every $n,\, m\in \ZZ_{>0} $ and $1\ls j\ls m$ with the splitting $\Phi_m^R$. Therefore,  the module $\HH{i}{\fm}{R/I^{(n+1)}}$ is a direct summand of $\HH{i}{\fm}{R^{1/m}/(I^{(nm+j)})^{1/m}}$ for every $1\ls i\ls\dim R/I$. 
We note that  
\begin{equation}\label{eqds}
\left( \HH{i}{\fm}{ R/I^{(nm+j)}\right)}^{1/m}=
\HH{i}{\fm^{1/m}}{R^{1/m}/(I^{(nm+j)})^{1/m}}
=\HH{i}{\fm}{R^{1/m}/(I^{(nm+j)})^{1/m}},
\end{equation}
because the ring isomorphism $R\cong R^{1/m}$ gives an equivalence of categories and $\fm R^{1/m}$ is $\fm^{1/m}$. 
We conclude that
$\HH{i}{\fm}{R/(I^{(nm+j)})}=0$
implies  $\HH{i}{\fm}{R^{1/m}/(I^{(nm+j)})^{1/m}}=0$, and then $\HH{i}{\fm}{R/I^{(n+1)}}=0$.
Therefore, $$\depth(R/I^{(n+1)})\gs \depth(  R/(I^{(nm+j)})),$$ which proves the first part.

From the splitting $\iota$ and Equation \eqref{eqds}, we have 
$$
a_i(R/I^{(n+1)})\ls a_i \left( R^{1/m}/(I^{(nm+j)})^{1/m}\right)=\frac{1}{m} a_i \left(  R/I^{(nm+j)}\right),
$$
and the second part follows.
\end{proof}

 As a consequence of Theorem \ref{ThmDivision} we recover the following limits for $a$-invariants of symbolic powers. 

\begin{thm}[{\cite[Theorem 4.7]{LT10}}]\label{limAinv}
Let $I$ be a square-free monomial ideal.
Then
$$
\lim\limits_{n\to\infty}\frac{a_i(R/I^{(n)})}{n}
$$
exists for every $0\ls i\ls \dim R/I$.
\end{thm}
\begin{proof}
Fix $i$.  The sequence $\{\frac{\reg(R/I^{(n)})}{n}\}_{n\gs 0}$ has an upper bound \cite[Corollary 3.3]{herzoga2007symbolic}, then so does $\{\frac{a_i(R/I^{(n)})}{n}\}_{n\gs 0}$. Set $\alpha_n = \frac{a_i (R/I^{(n)})}{n}$ for every $n\gs 0$ and $\alpha=\sup\{ \alpha_n\}$.

If $\alpha=-\infty$, we have that $\alpha_n=-\infty$ for every $n$ and the claim follows.  We now assume that $\alpha\neq -\infty$ and show that 
$\lim\limits_{n\to\infty}\alpha_n=\alpha.$ 
We note that Theorem \ref{ThmDivision}(2) implies 
\begin{equation}\label{useIneq}
\alpha_n\gs \frac{m}{n}\left\lceil\frac{n}{m}\right\rceil\alpha_{\lceil\frac{n}{m}\rceil}\qquad \text{ for every }n,m\in \ZZ_{>0}.
\end{equation}

Fix $\epsilon\in \RR_{>0}$ and  let $t\in \ZZ_{\gs 0}$ be such that $\alpha-\alpha_t<\epsilon.$ It suffices to show $\alpha_n> \alpha_t$  for every $n\gs t^2$ as this implies that $\alpha-\alpha_n<\epsilon$.  Let $n = t q + r$ for some $q, r\in \ZZ_{\gs 0}$  and $r\in [0,t-1]$,
then $n=(q+1)t-(t-r)$. 
Since $n\gs t^2$, we obtain  $q\gs t$ and then, $0< t-r\ls q-r<q+1,$ and so $\lceil\frac{n}{q+1}\rceil=t.$ Applying  inequality \eqref{useIneq} with $m=q+1$ we have $\alpha_n\gs \frac{tq+t}{tq+r}\alpha _t>\alpha_t$, which finishes the proof.
\end{proof}

As a corollary we obtain that the related limit for the Castelnuovo-Mumford regularity of symbolic powers exists. 

\begin{cor}[{\cite[Theorem 4.9]{LT10}}]\label{limReg}
Let $I$ be a square-free monomial ideal.
Then
$$
\lim\limits_{n\to\infty}\frac{\reg(R/I^{(n)})}{n}=\max\left\{\lim\limits_{n\to\infty}\frac{a_i(R/I^{(n)}) }{n}\right\}.
$$
\end{cor}
\begin{proof}
The result follows from Theorem \ref{limAinv} and the following equalities
\begin{align*}
\lim\limits_{n\to\infty}\frac{\reg(R/I^{(n)})}{n} &= \lim\limits_{n\to\infty}\frac{ \max\{a_i(R/I^{(n)}) + i\}}{n}
\\ &=  \lim\limits_{n\to\infty}\frac{ \max\{a_i(R/I^{(n)}) \}}{n} = \max\left\{\lim\limits_{n\to\infty}\frac{a_i(R/I^{(n)}) }{n}\right\}.
\end{align*}
\end{proof}

\begin{remark}\label{RemWC}
If $I$ is a square-free monomial ideal,
then
$$
\lim\limits_{n\to\infty}\frac{\reg(I^{(n)})}{n}=\lim\limits_{n\to\infty}\frac{\reg(R/I^{(n)})+1}{n}=\lim\limits_{n\to\infty}\frac{\reg(R/I^{(n)})}{n}.$$
Let $\alpha(I^{(n)})$ denote the smallest degree of a nonzero element of $I^{(n)}$. The 
{\it Waldschmidt constant}  is defined by 
$\widehat{\alpha}(I)=\lim\limits_{n\to\infty}\frac{\alpha(I^{(n)}) }{n}$ \cite{BocciEtAl}. Then 
$$
\widehat{\alpha}(I)\ls \lim\limits_{n\to\infty}\frac{\reg(I^{(n)})}{n}=\lim\limits_{n\to\infty}\frac{\reg(R/I^{(n)})}{n}.
$$

\end{remark}

We recall  that $\R^s(I)$ is a Noetherian algebra \cite[Theorem 3.2]{herzoga2007symbolic}. Therefore, $\reg(I^{(n)})$ agrees with a linear quasi-polynomial  $\gamma(n)  n + \theta(n)$ for $n\gg 0$.   As a consequence of Corollary \ref{limReg} we obtain that the leading coefficient of this quasi-polynomial is constant for square-free monomial ideals and we obtain a bound for $\theta(n)$. 

\begin{cor}\label{CorQP}\cite[Theorem 4.9]{LT10}
Let $I$ be a square-free monomial ideal. Then $\gamma(n) $ is equal to a constant $\gamma$ and $\theta(n)\ls \dim(R/I)+1$ for every  $n\gg 0$.
\end{cor}
\begin{proof}
 From Corollary \ref{limReg} it follows that $\gamma(n)$ must be equal to $\gamma:=\sup\{\frac{a_i(R/I^{(n)})}{n}\}$ for $n\gg 0$. 
For the second claim, we observe that
$$
\gamma n + \theta(n)=\reg(R/I^{(n)})+1=  \max\{a_i(R/I^{(n)})+i)\}+1\ls \gamma n+\dim(R/I)+1.
$$
for $n\gg 0$. We conclude that  $\theta(n)\ls \dim(R/I)+1$ for  every $n\gg 0$.
\end{proof}

The previous results, together with results for matroids \cite{RegSymbMatroit} and low dimension \cite{RegSymbDim2}, motivated Minh and  Trung \cite{RegSymbMatroit} to ask wheather  $\reg(R/I^{(n)})$  is a linear polynomial for $n\gg 0$ for square-free monomial ideals. We note that very recently a counter-example to this question was found by Dung,  Hien,  Nguyen, and  Trung \cite[Remark 5.16]{DHNT}.

It is a classical result that for any homogeneous ideal $I$,  $\reg(I^n)$ is a linear function $b(I)n+c(I)$ for $n\gg 0$ \cite{CHT,KV}. In general, not much is known about the invariant $c(I)$ besides the fact that it is non-negative \cite[Corollary 3.3]{trung2005asymptotic}. 
Corollary \ref{CorQP} provides an upper bound for $c(I)$ for a wide family of square-free monomial ideals.

\begin{cor}
Let $I$ be a square-free monomial ideal such that 
$I^n=I^{(n)}$ for every $n\gg 0$, then $c(I)\ls \dim (R/I)+1$. In particular, this holds for bipartite edge ideals.
\end{cor}
\begin{proof}
 The result follows by the assumption and Corollary \ref{CorQP}.  The case of edge ideals follows because 
 they satisfy $I^n=I^{(n)}$ for every $n\gg 0$ \cite[Theorem 5.9]{SVV}.
\end{proof}

\section{Associated Graded Algebras and Equality of Symbolic and Ordinary Powers}

In this section, we study the graded algebras defined in Definition \ref{blowup}.
This is in order to prove Theorem \ref{MainAllPowMon}. Our strategy is the following. We first show that the symbolic Rees and associated graded algebras split from their rings of $m$-roots in Theorem \ref{ThmGradedSpl}. Then in Theorem \ref{MainSQfree}, we characterize the equality of symbolic and ordinary powers in terms of this splitting. Finally, we use this characterization to prove Theorem \ref{MainAllPowMon}. We start with introducing rings of $m$-roots for the Rees and associated graded algebras.

 \begin{notation}\label{notBlw}
For $m\in \ZZ_{>0}$, we set  
$$\R(I)^{1/m}:=  \oplus_{n\gs 0} I^{n/m} t^{n/m}\subseteq A^{1/m}$$ 
and 
 $$\R^s(I)^{1/m}:=  \oplus_{n\gs 0} (I^{(n)})^{1/m} t^{n/m}\subseteq A^{1/m}.$$ 
 We consider the ideals 
$$\cJ(I)=\oplus_{n\gs 0} I^{n+1}t^n\subseteq \R(I)$$
 and
 $$\cJ^s(I)=\oplus_{n\gs 0} I^{(n+1)}t^n\subseteq \R^s(I).$$
 \end{notation}

A classical result states that $\gr_I(R)$ is reduced if and only if
$I^n=I^{(n)}$ for every $n\in\ZZ_{>0}$ \cite[Corollary 1.6]{herzog2008standard}. As a consequence of this result, and its proof, one has that $\cJ^s(I)$ is a radical ideal. We set 
$$\gr^s_{I} (R)^{1/m} = (\R^s(I)/\cJ^s(I))^{1/m}.$$

\begin{remark}
\

\begin{enumerate}
\item By Remark \ref{inclusionSymbol} we have the inclusion  
$$\R^s(I)=\oplus_{n\gs 0}  I^{(n)} t^n \subseteq \oplus_{n\gs 0}  \left( I^{(n)}\right)^{1/m} t^{n/m}= \R^s(I)^{1/m}.$$

\item Since $I\subseteq R$ is monomial, we have that $I^{n}\subseteq (I^{nm})^{1/m}$. 
Then $$\R(I)=\oplus_{n\gs 0}  I^n t^n \subseteq \oplus_{n\gs 0}  I^{n/m} t^{n/m}=\R(I)^{1/m}.$$
\end{enumerate}
\end{remark}

 \begin{thm}\label{ThmGradedSpl}
Let $I$  be a square-free monomial ideal. Then the maps induced by the observation  
in Remark \ref{inclusionSymbol}
  $$\R^s(I)\to \R^s(I)^{1/m}\quad\text{ and  }\quad\gr^s_{I} (R)\to \gr^s_{I} (R)^{1/m}$$ split for every $m\in \ZZ_{>0}$.
 \end{thm}
 \begin{proof}
Fix $m\in\ZZ_{>0}$ and let $\Phi^R_m$ be the splitting in  Definition \ref{defSplitting}. 
We define $$\varphi:\R^s(I)^{1/m}\to  \R^s(I)$$ to be the homogeneous morphism of $\R^s(I)$-modules  induced  by
 $\varphi(r^{1/m} t^{n/m}) =\Phi^R_m(r^{1/m})t^{n/m}$ if $m$ divides $n$, and $\varphi(r^{1/m} t^{n/m}) =0$ otherwise. 
The map $\varphi$ is well-defined because
$$
\Phi^R_m \big(( I^{((n+1)m)} )^{1/m}\big)  \subseteq   I^{(n+1)} 
$$
for every $n\gs 0$ by Lemma \ref{LemmaStable}, and it is $\R^s(I)$-linear since $\Phi^R_m$ is $R$-linear.
If $r\in I^{(n)}\subseteq ( I^{(nm)})^{1/m}$, then $\varphi(rt^{nm/m})=\Phi^R_m(r) t^n= rt^n$ because $\Phi^R_m$ is 
a splitting. We conclude that $\varphi$ is also a splitting.

Consider the ideal $\cJ=\cJ^s(I)$ as in Notation \ref{notBlw}. By Lemma \ref{LemmaStable}
  we obtain
$$\varphi\big(( I^{(nm+1)})^{1/m} t^{nm/m}\big)\subseteq  I^{(n+1)} t^n$$ for every $n\gs 0$.
Therefore, $\varphi(\cJ^{1/m})\subseteq \cJ$, i.e.,  $\cJ$ is compatible with
$\varphi$. This induces a splitting $\overline{\varphi}: (\R^s(I)/\cJ)^{1/m} \to \R^s(I)/\cJ$. The conclusion follows. 
 \end{proof}

Theorem \ref{ThmGradedSpl} has the following consequence if the field $k$ has positive characteristic.
 
 \begin{cor}\label{Fpure}
 Let $I$ be a square-free monomial ideal.
If  $k$ is a perfect field of  prime characteristic $p$, then $\R^s(I)$ and $\gr_I^s(R)$ are $F$-pure. 
\end{cor}
\begin{proof}
We note that the rings involved are $F$-finite. Then, they are $F$-pure if and only if they are $F$-split \cite[Corollary 5.3]{HoRo}.
Since $k$ is perfect, we have that $R^{1/p}$ and $\R^s(I)^{1/p}$ correspond to the rings of $p$-roots of $R$ and  $\R^s(I)$ respectively.  In addition, $(\cJ^s(I))^{1/p}$ is the ideal of $\R^s(I)^{1/p}$ that corresponds to 
$\cJ^s(I)$. Since  $\cJ^s(I)$ is radical,  $\gr_I^s(R)$ is a reduced ring. Then,  $\gr^s_{I} (R)^{1/p}$
corresponds to the ring of $p$-roots of $\gr^s_{I} (R)$.
Then, the result follows from  Theorem \ref{ThmGradedSpl} with $m=p$.
\end{proof}

The following Theorem provides necessary and sufficient conditions for the equality of ordinary and symbolic powers of square-free monomial ideals.
 
  \begin{thm}\label{MainSQfree}
Let $I$ be a square-free  monomial ideal.
Then the following are equivalent.
\begin{enumerate}
\item  $\Phi_m^R\big(( I^{nm+1})^{1/m}\big)\subseteq I^{n+1}$ for every $m\in\ZZ_{>0}$ and $n\in \ZZ_{\gs 0}$.
\item  $\Phi_m^R\big(( I^{nm+1})^{1/m}\big)\subseteq I^{n+1}$ for some  $m\in\ZZ_{>1}$ and every $n\in \ZZ_{\gs 0}$.
\item $I^n=I^{(n)}$ for every $n\in \ZZ_{>0}.$
\end{enumerate}
 \end{thm}
 \begin{proof}

Clearly  $(1)$ implies  $(2)$. 
We now assume $(2)$ and prove $(3)$.
We first prove that  $\cJ(I)$ is a radical ideal. 
Let $f\in I^n$ be a monomial such that $ft^n\in \sqrt{\cJ(I)}$ then there exists and integer $e\gs 0$ such that $(f t^{n})^{m^{e}}\in\cJ(I)$. We observe that 
 $f^{m^e}\in I^{nm^e+1}$. 
 We note that the assumption in $(2)$ implies $\Phi_{m}^R\big(( I^{nm^e+1})^{1/m}\big)\subseteq I^{nm^{e-1}+1}$.
 Therefore,  
 \begin{align*}
  f^{m^{e-1}}t^{nm^{e-1}} =\Phi_{m}^A((f^{m^e}t^{nm^e})^{1/m})) 
 &\in \Phi_{m}^A(( I^{nm^e+1} t^{nm^e})^{1/m})\\
 & = \Phi_{m}^R(( I^{nm^e+1} )^{1/m})t^{nm^{e-1}}\subseteq I^{nm^{e-1}+1}t^{nm^{e-1}}\subseteq \cJ(I).
 \end{align*}
A decreasing induction on $e$ shows $ft^n\in  \cJ(I)$.
Then, $\gr_I(R)=\R(I)/\cJ(I)$ is reduced.
As a consequence, $I^n=I^{(n)}$ for every $n\in\ZZ_{>0}$ \cite[Corollary 1.6]{herzog2008standard}.

Now, we assume that $(3)$.  By Lemma \ref{LemmaStable}, we have $\Phi_m^R\big(( I^{nm+1})^{1/m}\big)\subseteq I^{n+1}$ for every $n,\, m\in \ZZ_{>0}$, therefore $(1)$ follows.
 \end{proof}

We now state a  problem due to Huneke which was brought to our attention by  H\`a at the BIRS-CMO workshop on {\it Ordinary and Symbolic Powers of Ideals} at  Casa Matem\'atica Oaxaca in the summer of 2017. 
 
\begin{problem}[Huneke]\label{ProbHa}
Let $I$ be a square-free monomial ideal.
Find a number $N\in\ZZ_{> 0}$, in terms of $I$, such that $I^n=I^{(n)}$ for every $n\ls N$ implies 
 $I^n=I^{(n)}$ for every $n\in\ZZ_{>0}$
\end{problem}

 This problem is strongly related to Conforti-Cornu\'ejols  (Conjecture \ref{ConjPackingProb}). In fact, if $N=\height(I)$ satisfies the conclusion of Problem \ref{ProbHa}, then Conjecture \ref{ConjPackingProb}  follows \cite[Remark 4.19]{SurveySymbPowers}. H\`a  also asked for an optimal value for $N$. In Example \ref{ExSharp}, we prove that our bound is sharp.
 
 In general $N=\lfloor\frac{(d+1)^{(d+3)/2}}{2^d}\rfloor$ works for any $I$ \cite[Theorem 5.6]{herzoga2007symbolic}; and if one assumes  Conjecture \ref{ConjPackingProb}, then  $N=\lceil \frac{d+1}{2}\rceil$ would also work \cite[Remark 4.8]{HaMorey}. 
 As a consequence of our methods, in Theorem \ref{allPowMon} we provide a sharper  
  $N$ in terms of the number of generators of $I$. For the proof of this result, we need the following well-known lemma. We include its proof for the sake of completeness.
 
  We denote by $\mu(I)$ the minimal number of generators of $I$. If $I$ is generated by the monomials $\fx^{\fa_1},\ldots, \fx^{\fa_u}$, we denote by $I^{[m]}$ the ideal generated by $\fx^{m\fa_1},\ldots, \fx^{m\fa_u}$.

\begin{lemma}\label{Lemma Obs pe u(I)}
Let $I$ be a monomial ideal.
If $r\gs \mu(I)( m-1)+1$, 
then  $I^r= I^{r- m} I^{[m]}$.
\end{lemma}
\begin{proof}
Let $u=\mu(I)$ and $\fx^{\fa_1},\ldots, \fx^{\fa_u}$  a minimal set of generators of
$I.$
Let $\alpha_1,\ldots,\alpha_u$ be natural numbers such that 
$\alpha_1+\ldots+\alpha_u=r,$ 
then by assumption there must exist $\alpha_i$ such that $\alpha_i\gs m.$
Therefore, 
$$
\fx^{\alpha_1\fa_1}\cdots \fx^{\alpha_u\fa_u}=
\fx^{\alpha_1\fa_1}\cdots \fx^{(\alpha_i-m)\fa_i} \cdots \fx^{\alpha_u\fa_u}\fx^{\alpha_im}
\in I^{r-m}I^{[m]}. 
$$
This shows that $I^r\subseteq I^{r- m} I^{[m]}$.
To obtain the other containment, we observe that $I^{[m]}\subseteq I^m$.
\end{proof}
 
 \begin{thm}\label{allPowMon}
 Let $I$ be a square-free  monomial ideal.
If $I^n=I^{(n)}$ for every $n\ls \lceil\frac{ \mu(I)}{2}\rceil$, then $I^n=I^{(n)}$ for every $n\in\ZZ_{>0}$.
 \end{thm}
\begin{proof}
%
We show that $$\Phi_{2}^R\big(( I^{2n+1})^{1/2}\big)\subseteq I^{n+1}$$ for every $n\gs 0$  and then the result follows by Theorem \ref{MainSQfree}, $(2)\Rightarrow (3)$. By assumption and Lemma \ref{LemmaStable} this inclusion holds for $n< \lceil\frac{\mu(I)}{2}\rceil$, as for these values  $I^{(n+1)}=I^{n+1}$. We fix $n\gs  \lceil\frac{\mu(I)}{2}\rceil$. Then $$2n+1\gs \mu(I)+1$$ 
and hence 
$( I^{2n+1})^{1/2}=( I^{2(n-1)+1})^{1/2}I$ by Lemma \ref{Lemma Obs pe u(I)}. 
Therefore, by induction
$$
\Phi_{2}^R\big(  ( I^{2n+1})^{1/2}   \big)
=\Phi_{2}^R\big(( I^{2(n-1)+1})^{1/2} I   \big)
=\Phi_{2}^R\big(( I^{2(n-1)+1})^{1/2}       \big)I
\subseteq   I^{n}I=I^{n+1},
$$
finishing the proof.
\end{proof} 

We point out that Theorem \ref{allPowMon} relates to an open problem stated by Francisco, H\`a, and Mermin \cite[Problem 5.14(a)]{francisco2013powers}.
The following example shows that Theorem \ref{allPowMon} is sharp.

\begin{example}\label{ExSharp}
Let $R=k[x_1,\ldots, x_{2t-1}]$ for some $t\gs 2$, and let $$I=(x_1x_2,x_2x_3,\ldots, x_{2t-2}x_{2t-1}, x_{2t-1}x_{1}).$$ Then $I^{n}=I^{(n)}$ for every $n<t=\lceil\frac{\mu(I)}{2}\rceil$, whereas $I^{t}\neq I^{(t)}$ \cite[Corollary 4.5]{LamTrung2015}.
\end{example}

\begin{cor}\label{CorFrob}
Let $I$ be a square-free  monomial ideal.
Then 
$I^n=I^{(n)}$  for every $n\in\ZZ_{>0}$ 
if and only if $x_1\cdots x_d I^{2n+1}\subseteq \left( I^{n+1}\right)^{[2]}$ for $n< \frac{\mu(I)}{2}$. 
\end{cor}
\begin{proof}
Let $m\in \ZZ_{>0}$. Since $R^{1/m}$ and $R$ are regular of the same dimension, we have that
 $\Hom_R(R^{1/m}, R)\cong R^{1/m}$ as $R^{1/m}$-modules \cite[Lemma 1.6 (1)]{FedderFputityFsing}. Furthermore,  standard computations show that the map $\Psi_m :R^{1/m}\to R$ induced by 
$$
\Psi_m(\fx^{\fa/m})=
\begin{cases} 
     \fx^{(\fa - (m-1)\f1)/m} & \fa \equiv -\f1\,  (\bmod\, m);\\
      0 & \hbox{otherwise,}
\end{cases}
$$
where $\f1 = (1,\ldots, 1)\in \ZZ^d$, is a generator  of $\Hom_R(R^{1/m}, R)$ as an $R^{1/m}$-module (this is a standard computation in prime characteristic $p$ when $m=p$).
We now  focus on the case $m=2$. We stress that we are not making any assumption on the characteristic of the field.
We note that
$x_1\cdots x_d I^{2n+1}\subseteq \left( I^{n+1}\right)^{[2]}$
if and only if 
$$(x_1\cdots x_d)^{1/2} (I^{2n+1})^{1/2}\subseteq  I^{n+1}R^{1/2},$$
as $R^{1/2}$ is isomorphic to $R$.
In addition, $(x_1\cdots x_d)^{1/2} (I^{2n+1})^{1/2}\subseteq I^{n+1}R^{1/2}$ 
is equivalent to 
$\Psi_2\big((x_1\cdots x_d)^{1/2} (I^{2n+1})^{1/2}\big)\subseteq I^{n+1}$ for every $n\gs 0$
\cite[Lemma 1.6 (2)]{FedderFputityFsing}. 
Since $\Phi_2^R(-)=\Psi_2\big( (x_1\cdots x_d)^{1/2}-\big)$, the result follows from Theorem \ref{MainSQfree} and the proof of Theorem \ref{allPowMon}.
\end{proof}

 We note that Theorem \ref{allPowMon} and Corollary \ref{CorFrob} give an algorithm to verify Conjecture \ref{ConjPackingProb} for a specific ideal.
 \section{Applications to linear optimization}\label{SecOptimization}
 
 In this brief section we translate our result to the context of linear programming. For more on this topic, we refer to  \cite{HaTrung}.

 A {\it clutter} $\C=(V, E)$ is a collection of subsets $E$ of $V=\{v_1,\ldots, v_n\}$ such that every two elements of $E$ are incomparable with respect to inclusion. We denote by $M:=M(\C)$ the $n\times m$ matrix with entries equal to 0 or 1, such that its columns are the incidence vectors of the sets in $E$. Given $\fc\in \ZZ^n_{\gs 0}$, by the Strong Duality Theorem we have the following equality of dual linear programs.
 \begin{equation}\label{dualSystem}
  \min\{\fc\cdot \fx\mid \fx\in \RR^n_{\gs 0},\, M^T\fx\gs \f1_m\}=
 \max\{\f1_m\cdot \fy\mid \fy\in \RR^m_{\gs 0},\, M\fy\ls \fc\},
 \end{equation}
where $\f1_m=(1,\ldots, 1)\in \ZZ^m$. We say that $\C$ {\it packs for $\fc$} if Equation \eqref{dualSystem} holds for integer vectors $\fx \in \ZZ_{\gs 0}^n$ and  $\fy\in \ZZ_{\gs 0}^m$. The clutter $\C$ satisfies the {\it Max-Flow-Min-Cut (MFMC) property}  if it packs for every $\fc\in \ZZ_{\gs 0}^n$.

We now recall a lemma that allows us to rephrase Theorem \ref{allPowMon} in this context. We first set up some notation. Let $\{\fb_1,\ldots, \fb_m\}$ be the column vectors of $M$. We denote by $I$ the ideal of the polynomial ring $R=k[v_1,\cdots, v_n]$ generated by the monomials $\{\fv^{\fb_1},\ldots, \fv^{\fb_m}\}$.  
  
 \begin{lemma}[{\cite[Proposition 3.1]{Trung06}}]\label{translation}
  Set  $$\gamma(\fc)  := \min\{\fc\cdot \fx\mid \fx\in \ZZ^n_{\gs 0},\, M^T\fx\gs \f1_m\}, \quad \textnormal{ and }$$ 
$$ \sigma(\fc) := \max\{\f1_m\cdot \fy\mid \fy\in \ZZ^m_{\gs 0},\, M\fy\ls \fc\}.$$
 Then 
 \begin{enumerate}
 \item[(i)] $\fv^\fc\in I^{(t)}$ if and only if $t\ls \gamma(\fc)$.
\item[(ii)] $\fv^\fc\in I^{t}$ if and only if $t\ls \sigma(\fc)$.
 \end{enumerate}
 \end{lemma} 
 
 We are now ready to present the main theorem of this section which is related to previous results in integer programming \cite[Corollary 2.3]{BT82}. 

 \begin{thm}\label{MFMC}
 The clutter 
 $\C$ packs for every $\fc\ls \f1_n\lceil \frac{m}{2}\rceil$ if and only if $\C$ satisfies the MFMC property.
 \end{thm}

\begin{proof}
By definitions the MFMC property implies the $\C$ packs for every $\fc\in \ZZ_{\gs 0}^m$.  We focus on the other direction.
We assume $\C$ packs for every $\fc\ls \f1_n\lceil \frac{m}{2}\rceil$. Let $\fv^\fc$ be a minimal monomial generator of $I^{(t)}$  for some $1\ls t\ls \lceil\frac{m}{2}\rceil$, then it clearly  divides $\fv^{t\f1_n}$. Moreover, by Lemma \ref{translation}, we have $\gamma(\fc)=t$; this equality holds since reducing a component of $\fc$ by 1, reduces $\gamma(\fc)$ by at most 1. Therefore, by assumption we have $\sigma(\fc) = \gamma(\fc)= t$ 
and then $\fv^\fc\in I^t$ by Lemma  \ref{translation}. We conclude that  $I^{(t)}=I^t$ for $t=1,\ldots, \lceil\frac{m}{2}\rceil=\lceil\frac{\mu(I)}{2}\rceil$. By Theorem \ref{allPowMon} it follows that $I^{(t)}=I^t$ for every $t\gs 1$. Now,  let $\fc\in \ZZ_{\gs 0}^n$  be arbitrary, then $\fv^\fc\in I^{(\gamma(\fc))}=I^{\gamma(\fc)}$ and by Lemma \ref{translation}  we have $\sigma(\fc)\gs \gamma(\fc)$. Since $\sigma(\fc)\ls \gamma(\fc)$ always holds, it follows that $\sigma(\fc)=\gamma(\fc)$ and this finishes the proof.
\end{proof}  




\section*{Acknowledgments}
The authors started this project after participating in the BIRS-CMO workshop on {\it Ordinary and Symbolic Powers of Ideals} Summer of 2017,  Casa Matem\'atica Oaxaca, Mexico, where they learned new open problems on the subject. The authors thank the organizers:  Chris Francisco, Tai H\`a, and Adam Van Tuyl. They also thank Craig Huneke for suggestions that improved  Theorem \ref{allPowMon} and Rafael H. Villarreal for suggesting Remark \ref{RemWC}. The authors would like to thank Hailong Dao, Elo\'isa Grifo,   Tai H\`a, and Robert Walker for suggestions on an earlier draft. Part of this project was completed in the Mathematisches Forschungsinstitut Oberwolfach (MFO) while the first  author was in residence at the institute under the program  {\it Oberwolfach Leibniz Fellows}. The first  author thanks MFO for their hospitality and excellent conditions for conducting research. The second author thanks Jack Jeffries for inspiring conversations. The authors would like to thank the referees for their helpful comments and suggestions which improved this paper.  

\bibliographystyle{plain}
\bibliography{References}

\end{document}